\documentclass[11pt]{amsart}
\usepackage{amsfonts, amsmath, amssymb, amscd, amsthm, graphicx, setspace, enumitem, color}
\usepackage{hyperref}
\usepackage{pinlabel}
\usepackage{mathtools}
\usepackage{graphicx}
\usepackage{stackengine}
\usepackage{extarrows}
\usepackage{bm}

\newtheorem{theorem}{Theorem} [section]
\newtheorem{lemma}[theorem]{Lemma}
\newtheorem{proposition}[theorem]{Proposition}
\newtheorem{corollary}[theorem]{Corollary}

\newtheorem{example}[theorem]{Example}

\newtheorem{remark}[theorem]{Remark}

\theoremstyle{definition}
\newtheorem{definition}[theorem]{Definition}
\newtheorem{notation}[theorem]{Notation}

\usepackage{tikz}
\usetikzlibrary{decorations.markings}
\usetikzlibrary{arrows,snakes}
\usepackage{verbatim}
\tikzset{edge/.style = {->,> = latex}}
\usetikzlibrary{decorations.pathreplacing,shapes.multipart}
\usetikzlibrary{automata} 
\usetikzlibrary{positioning} 
\usetikzlibrary{arrows} 

\newcommand\N{{\mathbb N}}
\newcommand\Z{{\mathbb Z}}

\newcommand\Q{{\mathbb Q}}

\newcommand{\NF} {\mathrm{NF}}
\newcommand{\GF} {\mathrm{GF}}
\newcommand{\cA} {\mathcal A}

\newcommand{\mbS}{\mathbb{S}}
\newcommand{\T}{\intercal}

\newcommand{\fsa}{{\mathrm{fsa}}}

\title{Rational sets in virtually abelian groups: languages and growth}
\author{Laura Ciobanu and Alex Evetts}
\address{Maxwell Institute of Mathematical Sciences, and Department of Mathematics, Heriot-Watt University,
Edinburgh EH14 4AS, UK}
\email{l.ciobanu@hw.ac.uk}
\address{Heilbronn Institute for Mathematical Research, and Department of Mathematics, University of Manchester, Manchester M13 9PL, UK}
\email{alex.evetts@manchester.ac.uk}

\begin{document}

\begin{abstract}
In this paper we generalise and unify the results and methods used by Benson, Liardet, Evetts, and Evetts \& Levine, to show that rational sets in a virtually abelian group $G$ have rational (relative) growth series with respect to any generating set for $G$. We prove equivalences between the structures used in the literature, and establish the rationality of important classes of sets in $G$: definable sets, algebraic sets, conjugacy representatives and coset representatives (of any fixed subgroup), among others.
Furthermore, we show that any rational set, when written as words over the generating set of $G$, has several EDT0L representations. 	
\end{abstract}

\maketitle

\section{Introduction}

 In this paper we present effective ways of representing sets of group elements in finitely generated virtually abelian groups as formal languages. Given a finite set $A$, any collection of words over $A$ is a called a formal language, and can be identified with a subset of the free monoid $A^*$. Formal languages are classified in the literature in terms of their complexity, and our goal is to find the simplest possible classes of languages that can encode the structure of sets of group elements: to this end, the main protagonists of the paper will be `regular languages' and the related `rational sets', but other types of sets and languages will feature as well, such as semilinear, $n$-regular or EDT0L.

All groups in this paper are finitely generated. For a group $G$ with a finite generating set $S$, let $\varpi: S^* \mapsto G$ be the natural projection which maps words to group elements. A subset $R\subseteq G$ is called \emph{rational} if $R=\varpi(L)$ for some regular language $L$ over $S$, where by \emph{regular language} we mean a set of words over $S$ defined by some finite state automaton (see Section \ref{sec:languages}). A useful feature here is that the rationality of a set is independent of the generating set (see \cite[Section 3]{LohreySurvey}).

Rational sets have played an important role in group theory since the 1990s (see \cite{bartholdi2010rational}), especially in free and hyperbolic groups. The notions of `regular' and `rational' coincide in free monoids by definition, but are different otherwise. Rational sets in groups consist of group elements, while regular sets consist of words in a free monoid. Moreover, in general, there is no bijection between rational sets of elements and regular languages of words because rational sets do not always have a regular language of unique representatives in the group. Nevertheless, being able to produce a set in a group via a finite state automaton is useful for the design of algorithms (\cite{groups_langs_aut}), can help compute the growth series of the set, or establish the series' rationality. Another concept related to rationality is \emph{semilinearity} (see Section \ref{sec:ratsemilin} for definitions and properties), and for subsets of commutative monoids, semilinearity coincides with rationality (Theorem \ref{thm:ES}(1)). In the most basic and well known setting, where the monoid is $\mathbb{N}^p$, a \emph{linear} set is of the form $\{v+k_1u_1 +k_2u_2 +\dots+k_mu_m \in \mathbb{N}^p \mid k_i \in \mathbb{N} \ \textrm{for} \ 1\leq i \leq m\}$, where $m \in \mathbb{N}$ and $v, u_i$ are vectors in $\mathbb{N}^p$; a \emph{semilinear} set is a finite union of linear ones.

We focus in this paper on finitely generated virtually abelian groups, and show that a multitude of sets in these groups turn out to be rational (Theorem \ref{Thm:growth}(2)); as a result, the sets have a nice characterisation as formal languages in terms of both natural normal forms (Theorem \ref{Thm:introNF} and paragraph above it) and geodesic representatives (Theorem \ref{Thm:geod_reps}). That is, we show that rational sets can be expressed in terms of $n$-regular and EDT0L languages, and have rational weighted growth series with respect to \emph{any} generating set and \emph{any} weight (Theorem \ref{Thm:growth}(2)). In fact, our results can be extended to \emph{complete growth series} that arise in the group ring $\mathbb{Z}[G]$ (see \cite{Liardet}), but we do not consider this generalisation in the paper.

A \emph{geodesic representative}, or simply \emph{geodesic}, for an element $g$ in $G$ over the finite generating set $S$, is a word $w$ such that the length of $w$ is minimal among all words $v$ over $S$ that represent $g$, that is, for which $\varpi(v)=g$. Geodesic representatives for group elements are essential, since they provide the lengths of groups elements, which are needed to compute growth. It is well known that the growth series of a regular language is a rational function, and so a standard technique in group theory is to exhibit a regular language of unique geodesic representatives for a set of group elements in order to compute its growth series and show this series is rational. However, the sets studied in this paper are not always expressible as  regular languages of geodesics, and different approaches to growth are needed for virtually abelian groups; these  go all the way to Benson (\cite{Benson}), and the unpublished thesis of Liardet (\cite{Liardet}), where polyhedral and semilinear sets, respectively, are used to establish rationality of the growth series. While Benson and Liardet's proofs have a common thread, the structures they rely on appear to be different at first.

We show here that the types of sets used in \cite{Benson} and \cite{Liardet} are in fact identical, and coincide with the \emph{coset-wise polyhedral} sets in the work of the second author and Levine on algebraic sets in virtually abelian groups. Furthermore, there is yet another connection between the various types of sets studied here and one more kind of sets: the subsets of $\mathbb{N}^k$ defined by \emph{Presburger formulas} and called \emph{Presburger sets}. From \cite[p. 297 - 298]{GS66} it follows that any polyhedral set of $\mathbb{N}^k$ is in fact a Presburger set, and from \cite[Theorem 1.3]{GS66} one can see that the class of Presburger sets is in fact the same as the class of semilinear sets of $\mathbb{N}^k$. It was shown in \cite[Theorem 1.5]{Woods15} that a subset $U \subseteq \mathbb{N}^k$ has rational multivariate generating function (of a particular kind) if and only if $U$ is a Presburger set. This shows how natural Presburger, and therefore semilinear, sets are in the study of rational generating functions.

Besides unifying the approaches of Benson \cite{Benson}, Liardet \cite{Liardet}, Evetts \cite{Evetts}, and Evetts \& Levine \cite{EvettsLevine} for studying the growth of sets in virtually abelian groups, we add to the list of rational sets further important classes: definable (by first-order theory) sets, algebraic sets (that is, solution sets to system of equations), conjugacy representatives, and coset representatives of any given subgroup. Definable sets are rational because they are Boolean combinations of cosets of definable subgroups of virtually abelian groups \cite[Theorems 4.1 and 3.2]{HrushovskiPillay}. All algebraic sets are definable since they are determined by existential first-order formulas, but we mention them separately as they are important on their own, and because they were the object of study in \cite{EvettsLevine} while definable sets were not.

Conjugacy and coset representatives were studied extensively in the work of the second author (\cite{Evetts}), where an approach to choosing one shortest element and an appropriate geodesic representative per conjugacy class or coset was given, and where it was shown that the conjugacy growth series and the growth series of coset representatives (of any fixed subgroup) are rational. Based on the work of the second author in \cite{Evetts} we are able to show here that a set of conjugacy representatives can be chosen to form a rational set, and similarly for a set of coset representatives for any fixed subgroup.

We summarise the results in this paper, which include the ones mentioned above due to several authors, in Theorems \ref{Thm:growth} and \ref{Thm:structure}. Our results for growth hold for any (positive integer) weight we assign to the group generators (see Section \ref{sec:growth}); when each generator has weight equal to $1$ we obtain the standard growth series.  

\begin{theorem}[Theorems \ref{thm:cosetreps}, \ref{thm:conjreps}, \ref{thm:CWPrational}, Corollary \ref{cor:rational-rational}]\label{Thm:growth} \leavevmode

Let $G$ be a finitely generated virtually abelian group.
\begin{enumerate}
\item A rational set in $G$ has rational weighted growth series with respect to \underline{any} generating set of $G$ and \underline{any} weight on $G$.  

\item The following types of sets are rational, and therefore have rational weighted growth series with respect to \underline{any} generating set of $G$ and \underline{any} weight on $G$:
\begin{enumerate}
\item elements of any fixed subgroup,
\item coset representatives of a fixed subgroup,
\item algebraic sets,
\item definable sets,
\item conjugacy representatives (as in Theorem \ref{thm:conjreps}).
\end{enumerate}
\end{enumerate}
\end{theorem}

 Theorem \ref{Thm:growth} follows largely from Theorem \ref{Thm:structure}, which explains the connections between the types of sets studied in the literature in order to establish rationality (see Section \ref{sec:eqv_sets} for the definitions of `polyhedral', `coset-wise polyhedral' and `semilinear').

\begin{theorem}\label{Thm:structure} \leavevmode 
\begin{enumerate} 
\item The following implications hold for subsets of $\mathbb{Z}^n$, $n\geq 1$:
\[\operatorname{polyhedral} \xLongleftrightarrow{\text{Proposition \ref{Prop:PeqS}}}  \operatorname{semilinear}.\]
\item  The following implications hold for subsets of virtually abelian groups:

(i) $\operatorname{coset-wise \ polyhedral} \xLongleftrightarrow{\text{Proposition \ref{prop:rationalCWP}}}  \operatorname{rational} \xLongleftrightarrow{\text{\cite[Thm. 4.1.5]{Liardet}}} \Z^k\operatorname{ - \ semilinear},$

(ii) $\operatorname{definable} \xLongrightarrow{\text{Corollary \ref{Cor:DefRat}}}  \operatorname{rational}.$

\end{enumerate}
\end{theorem}

The rationality of the growth series in Theorem \ref{Thm:growth}, starting with Benson's result about the rationality of the group growth series, is remarkable because it holds for all generating sets, which is a rare behaviour for groups in general. The proofs showing rationality rely on the existence of geodesic representatives (with respect to any generating set) with rational growth series, which we establish in Section \ref{sec:gnf}. Moreover, these representatives have a nice structure from a language theoretic point of view, as in Theorem \ref{Thm:geod_reps}. Here $n$-regular languages (Definition \ref{def:fsa}) are a `higher-dimensional' version of standard regular languages; they are sets of tuples that we call $n$-variable languages, and they exhibit a lot of the properties of standard regular languages, but not all (see Remark \ref{rem:n-reg}(2)). Most importantly for this paper, they can be easily viewed as EDT0L languages by converting tuples of words into words (as in Proposition \ref{prop:forgetful}). EDT0L languages (see Section \ref{sec:EDT0L}) have proved to be a natural class of languages to represent, as words, important classes of sets in groups (\cite{eqns_free_grps, eqns_hyp_grps, VF_eqns}), and we recommend \cite{appl_L_systems_GT} for background and motivation on these languages and their applications to group theory.

For the following result, Proposition \ref{prop:forgetful} explains how representatives of group elements, given as tuples of words, project to words that form EDT0L languages.

\begin{theorem}[Theorem \ref{thm:GF}]\label{Thm:geod_reps}
In a virtually abelian group $G$, with respect to any finite set $X$ of generators, the following hold: (i) each of the sets (1) -- (6) can be seen as a set of tuples that form an $n$-regular language, for some $n$ depending on the generating set, (ii) each tuple corresponds to a geodesic representative of a group element (as in Notation \ref{not:GF}), and (iii) each set projects (as in Proposition \ref{prop:forgetful}) to an EDT0L language:
\begin{enumerate}
\item rational sets,
\item subgroups,
\item coset representatives of a fixed subgroup,
\item algebraic sets,
\item definable sets,
\item conjugacy representatives.
\end{enumerate}
\end{theorem}

However, the geodesic representatives above (in tuple or EDT0L form) are complicated to describe and impractical, so the second theme of the paper is to consider the natural normal forms over standard generating sets: words of the form $wt$, where $w$ represents an element in the finite index abelian subgroup, and $t$ is one of finitely many coset representatives. We show that if we choose to work with the natural normal form (which might not be geodesic) of a virtually abelian group, then any set in Theorem \ref{Thm:growth} will have an EDT0L representation as words over the generating set. 

\begin{theorem}[Theorem \ref{thm:NF}, Corollary \ref{cor:NFrational}] \label{Thm:introNF}
Let $G$ be a virtually abelian group with set of generators $X=\Sigma \cup T$, where $\Sigma$ is a symmetric set of generators for a finite index free abelian subgroup of $G$, and $T$ is a finite transversal set for this subgroup. 

The following have an $n$-regular natural normal form when viewed as tuples of words, and therefore project to EDT0L languages, with respect to $X$:
\begin{enumerate}
\item rational sets,
\item subgroups,
\item coset representatives of a fixed subgroup,
\item algebraic sets,
\item definable sets,
\item conjugacy representatives.
\end{enumerate}
\end{theorem}

We note that Bishop (\cite{Bishop}) has investigated both the formal language properties and the growth series of the set of \emph{all} geodesic words in a virtually abelian group, showing amongst other results that this set forms a \emph{blind multicounter} language. Although it seems unlikely that the language is also EDT0L, it is not impossible, and we do not have enough evidence to put forward a conjecture in either direction. Furthermore, Bishop (\cite{BishopCW}) recently studied the coword problem in virtually abelian groups, that is, the set of words representing the trivial element, and showed that the cogrowth series of a virtually abelian group is the diagonal of an $\mathbb{N}$-rational series for every finite monoid generating set.

\section{Preliminaries}\label{sec:prelims}

\subsection{Growth of sets and groups}\label{sec:growth}
We will work with the basis $\{e_1,\ldots, e_r\}$ of $\Z^r$, where $e_i$ denotes the standard basis vector with $1$ in the $i$th entry and zeroes elsewhere. 

Let $U \subseteq\Z^r$ be a set. Given some choice of
	weight function $\|e_i\|\in\Z_{>0}$ (typically $\|e_i\|=1$ for all $i$) for the basis vectors $\{e_i\}_{i=1}^r$, we assign the $\ell_1$ norm $\| \mathbf{p}\|=\sum_{i=1}^r a_i\|e_i\|$ to the element
	$\mathbf{p}=(a_1, \ \ldots, \ a_r)\in U$. We define the spherical growth function counting elements of specified weights as
	\[\sigma_U(n)=\#\{\mathbf{p}\in U\mid\|\mathbf{p}\|=n\},\] and the resulting
	weighted growth series as
	\[\mbS_U(z)=\sum_{n=0}^\infty\sigma_U(n)z^n.\] 
	
	More generally, suppose $G$ is a group generated by a finite set $X$, with weights \(\|x\|\) for each generator (equal to 1 in the usual case). If $w=x_1\cdots x_k\in X^*$ is a word in the generators, we write
	$\overline{w}\in G$ for the group element that the word $w$ represents instead of the more cumbersome projection notation $\varpi(w)$ (as in second paragraph of the paper), and define the weight of \(w\) in the natural way as \(\|w\|=\sum_i\|x_i\|\). Let $\|g\|_X=\mathrm{min}\left\{\|w\|\mid w\in X^*,~\overline{w}=g\right\}$ be the weight of $g$ with respect to $X$. When all the weights are equal to 1, this is the familiar notion of word length. The (relative) weighted growth function of any set $U \subseteq G$ is $\sigma_{U,X}(n)=\#\{g\in U\mid ||g||_X=n\},$ and the (relative) weighted growth series is given by \[\mbS_{U,X}(z)=\sum_{n=0}^\infty\sigma_{U,X}(n)z^n.\] 

	We will frequently be interested in proving that the various growth series are rational: that is, there exist polynomials \(p,q\in\Z[z]\) such that \(\mbS(z)=\frac{p(z)}{q(z)}\). Moreover, the growth series will frequently lie in the following more restrictive class of \emph{\(\N\)-rational functions}, which have attracted recent attention in relation to growth of groups (see \cite{Bodart}).
	\begin{definition}\label{def:Nrational}
		The set of \emph{\(\N\)-rational functions} is the smallest set of functions $f(z)$ containing the polynomials \(\N[z]\) and closed under addition, multiplication, and quasi-inverse (that is, if \(f(0)=0\) then the quasi-inverse of \(f(z)\) is defined as \(\frac{1}{1-f(z)}\)).
	\end{definition}
	
\subsection{Preliminaries on formal languages} \label{sec:languages}
We start by introducing $n$-regular languages, $n\geq 1$, which are a generalisation of standard regular languages; that is, for $n=1$ we get exactly the regular languages defined by finite state automata (fsa).

\subsubsection{$n$-regular languages}
Here we take the approach from \cite[Section 2.10]{groups_langs_aut}, but slightly change the terminology to simplify the wording in this paper.
Let $A$ be a set. We write $A^n$ for the cartesian product $A \times \cdots \times A$ of $n$ copies of $A$, and write $\varepsilon$ for the empty word.

\begin{definition}\label{def:fsa}
	An \emph{asynchronous, $n$-variable finite state automaton} ($n$-variable fsa) $\mathcal{A}$ is a tuple $(\Sigma,\Gamma,s_0,F)$, where
	\begin{enumerate}
		\item $\Sigma$ is a finite alphabet,
		\item $\Gamma$ is a finite directed graph with edges labelled by elements of $(\Sigma\cup \{\varepsilon\})^n$ which have at most one non-$\varepsilon$ entry,
		\item $s_0\in V(\Gamma)$ is a chosen start vertex/state,
		\item $F\subseteq V(\Gamma)$ is a set of accept or final vertices/states.
	\end{enumerate}
	An $n$-tuple $\textbf{w}\in(\Sigma^*)^n$ is accepted by $\mathcal{A}$ if there is a directed path in $\Gamma$ from $s_0$ to some $s \in F$ such that $\textbf{w}$ is obtained by concatenating the labels on the path, and deleting all occurrences of $\varepsilon$ in each coordinate. A set of tuples like $\textbf{w}$ will be called an $n$-\emph{variable language}. An $n$-\emph{variable language} that is accepted by an asynchronous $n$-variable fsa will be called \emph{$n$-regular} for short.
\end{definition}

\begin{example}\label{ex:fsa}
The $2$-variable language $$L=\{(u,v) \mid u,v \in \{a,b\}^*, u, v \textrm{ have the same number of } a\}$$ is $2$-regular and given by the $2$-variable $\fsa$ in Figure \ref{fig:2fsa}. Here and later in the paper, we indicate the states by circles, the transitions by arrows, and distinguish the start state with an arrow, and the accept state(s) with a double circle.
\end{example}
\begin{figure}[h!]
		\centering
		\includegraphics[width=0.3\textwidth]{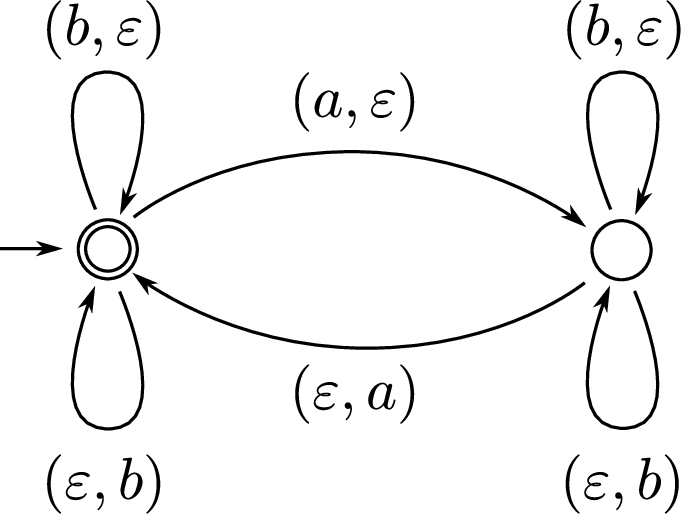}
		\caption{$2$-variable automaton for Example \ref{ex:fsa}}
		\label{fig:2fsa}
	\end{figure}

\begin{remark}\label{rem:n-reg}\leavevmode
\begin{enumerate}
\item[(1)] Definition \ref{def:fsa} gives a non-deterministic asynchronous $\fsa$ (see \cite[Section 2.5.1]{groups_langs_aut}), in that (i) edges labelled by tuples consisting entirely of $\epsilon$-coordinates are allowed, as are (ii) two edges with the same label starting at a single vertex. If we do not allow (i) and (ii) we have a \emph{deterministic} asynchronous $n$-$\fsa$.  
\item[(2)] Definition \ref{def:fsa} coincides with the usual definition of a regular language when $n=1$. However, for $n>1$, the classes of $n$-variable languages accepted by deterministic and non-deterministic asynchronous $n$-variable $\fsa$ do not coincide. See \cite[Exercise 2.10.3]{groups_langs_aut} for an example. 
\end{enumerate}
\end{remark}

Since $n$-variable fsa are not as common as $1$-variable fsa, we state and prove the following basic lemmas.
\begin{lemma}\label{lem:regularsubcat}
	The class of $n$-regular languages is closed under substitution and concatenation.
\end{lemma}
\begin{proof}
Suppose $f:\Sigma \rightarrow \Sigma^*$ is a substitution and $L$ is an $n$-regular language on alphabet $\Sigma$ recognised by  $n$-variable $\fsa$ $\mathcal{A}$. To obtain an $n$-variable fsa for $f(L)$ we replace each edge $e_{(\varepsilon, \dots, x, \dots, \varepsilon)}$, where $x\in \Sigma$, by $|f(x)|$ edges, each labelled with the appropriate  letter from $f(x)$ in the same coordinate as $x$ in $e_{(\varepsilon, \dots, x, \dots, \varepsilon)}$, while also inserting $|f(x)|-1$ vertices/states connecting the new edges. It is immediate to see that this new fsa accepts $f(L)$.

To obtain an $n$-variable $\fsa$ for the concatenation of two $n$-regular languages $L_1$ and $L_2$ recognised by $\mathcal{A}_1$ and $\mathcal{A}_2$, respectively, we simply connect each accepting state in $\mathcal{A}_1$ to the start state of $\mathcal{A}_2$ by an edge with $\varepsilon$-coordinates only.
\end{proof}

When considering tuples in $n$-variable languages or operations with such languages, it is a standard technique to `pad' shorter words so that the all tuples have the same number of coordinates (see \cite[Section 2.10.1]{groups_langs_aut})). In our case, for tuple of words over an alphabet $\Sigma$, we use the enlarged alphabet $\Sigma \cup \{\varepsilon\}$ and extend any $n$-tuple to an $m$-tuple, for $m>n$, by $m-n$ coordinates that take the value $\varepsilon$. Thus to any $n$-variable language $L$ one can associate an $m$-variable padded version $L^{\mathtt{P}}$, for $m>n$.  

\begin{lemma}\label{lem:regularunion}
	If $L_1\subset(\Sigma_1^*)^{n_1}$ is $n_1$-regular and $L_2\subset(\Sigma_2^*)^{n_2}$ is $n_2$-regular, with $n_1<n_2$, then the padded $L_1^{\mathtt{P}}\cup L_2\subset((\Sigma_1\cup\Sigma_2)^*)^{n_2}$ is $n_2$-regular. 
\end{lemma}
\begin{proof}
Suppose $L_1$ and $L_2$ are recognised by $\mathcal{A}_1$ and $\mathcal{A}_2$, respectively, and $n=n_1>n_2$. Then, as explained above the lemma, we pad all edges in $\mathcal{A}_2$ with $\varepsilon$ to have $n$ coordinates in both automata. We then add a new start state $s_n$ that connects to the start states of $\mathcal{A}_1$ and $\mathcal{A}_2$ via fully $\varepsilon$-edges with $n$ coordinates, while also adding a new final state $s_f$, to which all final states in $\mathcal{A}_1$ and $\mathcal{A}_2$ connect via fully $\varepsilon$-edges with $n$ coordinates. This new automaton, with single start state $s_n$ and single final state $s_f$, will accept $L_1\cup L_2\subset((\Sigma_1\cup\Sigma_2)^*)^{n}$.
\end{proof}

\subsubsection{EDT0L languages}\label{sec:EDT0L}

 In the 1960s, Lindemayer introduced a collection of classes of languages called
\textit{L-systems}, which were originally used for the study of growth of
organisms. EDT0L (\textbf{E}xtended \textbf{D}eterministic
\textbf{T}able \textbf{0}-interaction \textbf{L}indenmayer) languages are one of the L-systems, and were introduced by
Rozenberg in 1973 (see \cite{appl_L_systems_GT} for details). 

L-systems, including EDT0L
languages, were studied intensively in computer science in the 1970s
and early 1980s; L-systems are subclasses of indexed languages, which although not historically part of the Chomsky hierarchy, fit nicely between context-free and context-sensitive. Since the first author, Diekert and Elder's use of EDT0L languages to
study equations in free groups, a number of other works have used EDT0L
languages (or a similar class called ET0L languages)
\cite{appl_L_systems_GT} to describe a variety of sets in many important classes of groups.

\begin{definition}
	An \textit{EDT0L system} is a tuple $\mathcal{H}=(\Sigma, C, w_0, \mathcal{R})$, where
	\begin{enumerate}
		\item $\Sigma$ is a finite alphabet, called the \emph{terminal alphabet};
		\item $C$ is a finite set containing $\Sigma$, called the \emph{extended alphabet} of $\mathcal H$;
		\item $w_0\in C^*$ is called the \emph{start word};
		\item $\mathcal{R}$ is a rational subset of the monoid $\mathrm{End}(C^*)$ of endomorphisms of $C^*$; that is, there is some finite set $B\subset\mathrm{End}(C^*)$ so that $\mathcal{R}$ can be obtained as the image of a regular language over $B$.	
	$\mathcal{R}$ is called the \emph{rational control} of the EDT0L system.
	\end{enumerate}
	The language \textit{accepted} by \(\mathcal H\) is
	\[
	L(\mathcal H) = \{\phi(w_0) \mid \phi\in\mathcal{R}\}\cap \Sigma^*.
	\]
	A language accepted by an EDT0L system is called an \textit{EDT0L
		language}.
\end{definition}

We recommend the surveys \cite{appl_L_systems_GT} and \cite{GAGTA_chapter} on EDT0L languages and various aspects of group theory for more details, as well as lots of examples.

\begin{example}\label{eg:AlexL} \cite[Example 3.3]{GAGTA_chapter}
The language $L=\{a^{n^2} \mid n\in \mathbb{N}_+\}$ over the alphabet $\Sigma=\{a\}$ is EDT0L. 
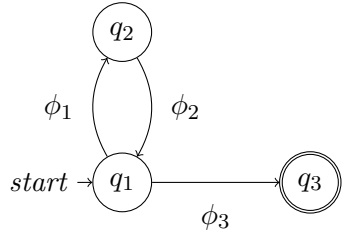
\begin{figure}[!h]
    \centering
   \begin{tikzpicture}  [node distance=2.5cm, scale=.4, every node/.style={circle}]

        \node[draw, initial] (q1)  {\(q_1\)};
        \node[draw, above of=q1,  yshift=-5mm] (q2) {\(q_2\)};
        \node[draw, right of=q1,accepting] (q3)  {\(q_3\)};

\draw[left,->,bend left] (q1) edge node {$\phi_1$} (q2);
\draw[right,->,bend left] (q2) edge node {$\phi_2$} (q1);
\draw[below,->] (q1) edge node {$\phi_3$} (q3);

      \end{tikzpicture}

    \caption{Finite state automaton defining the rational control for $L=\{a^{n^2}\}$}.
    \label{fig:AlexL}
\end{figure}
The extended alphabet is
     $ C = \{ s,  t,  u,  a\}$, the start word is $w_0=tsa$, and the finite set $B=\{\phi_1, \phi_2, \phi_3: C^* \to C^*\}$ is given by
     \[\begin{array}{lll}
\phi_1(s)= su,\\
\phi_2 (t) = at, \phi_2(u)=ua^2,\\
 \phi_3(s) =\phi_3(t)=\phi_3(u)= \varepsilon, \end{array}\] where we use  the convention that $\phi_i$ fixes the elements in $C$ not explicitly specified. The rational control $\mathcal{R}$ is given by the automaton $M$ in Figure \ref{fig:AlexL}, so $\mathcal{R}=L(M) =(\phi_1 \phi_2)^\ast \phi_3$. 
One can check that $(\phi_1\phi_2)^i(tsa)=a^itsua^2ua^4u\dots ua^{2i}a$, which is sent to $a^{(i+1)^2}$ by applying $\phi_3$. It follows that 
the language of the EDT0L system above is $\{a^{n^2} \ | \ n \in \mathbb{N}_+\}$. 

The language $L$ is not context-free, which can be shown by applying the Pumping Lemma for context-free languages (see \cite[Theorem 2.6.17]{groups_langs_aut}).
\end{example}

EDT0L languages have a number of useful closure properties (see for example \cite[Lemma 2.15]{EDT0L_extensions}). Here we only need the following result.

\begin{lemma}\cite[Lemma 2.15]{EDT0L_extensions}\label{lem:EDT0Lclosure}
	 EDT0L languages are closed under concatenation and finite union.
\end{lemma}

The following result is similar to Lemma 2.20(1) of \cite{EvettsLevine}, and it provides a bridge between $n$-regular and EDT0L languages via two natural maps.
\begin{proposition}\label{prop:forgetful}
	Let $L\subset(\Sigma^*)^n$ be an $n$-variable language and suppose $L$ is $n$-regular.
\begin{itemize}
\item[(i)] If $\theta\colon(\Sigma^*)^n\to\Sigma^*;~(w_1,w_2,\ldots,w_n)\mapsto w_1w_2\cdots w_n$ is the map that `forgets' the tuple structure, then $\theta(L)\subset\Sigma^*$ is EDT0L.
	
\item[(ii)]	Similarly, if $\theta_{\#}\colon(\Sigma^*)^n\to(\Sigma\cup {\#})^*;(w_1,w_2,\ldots,w_n)\mapsto w_1\#w_2\#\cdots \#w_n$ is the map that inserts $\#$ between the tuple coordinates, then $\theta_{\#}(L)\subset (\Sigma \cup {\#})^*$ is EDT0L.
\end{itemize}
\end{proposition}
\begin{proof}
	Let $(\Sigma,\Gamma,s_0,F)$ be an asynchronous $n$-variable fsa that accepts $L$ (possibly non-deterministic). 
	
	To prove (i) we construct an EDT0L system $\mathcal{H}$ that accepts $\theta(L)$ with terminal alphabet $\Sigma$, extended alphabet $C=\Sigma\cup\{\perp_1,\perp_2,\ldots,\perp_n\}$, and start word $\perp_1\perp_2\cdots\perp_n$.
	For each $\mathbf{w}=(w_1,\ldots,w_n)\in(\Sigma\cup\{\varepsilon\})^n$ which is a label of an edge in $\Gamma$, define $\phi_\mathbf{w}\in\mathrm{End}(C^*)$ by $\perp_i\mapsto w_i\perp_i$ for each $i$ (and fixing all elements of $\Sigma$). Furthermore, define $\varpi\in\mathrm{End}(C^*)$ to be the map that sends each $\perp_i$ to the empty word and fixes each element of $\Sigma$.
	
	Construct the rational control $\mathcal{R}$ of $\mathcal{H}$ from $\Gamma$ as follows. Replace each edge label $\mathbf{w}$ with the endomorphism $\phi_{\mathbf{w}}$. Add additional edges from $F$, each labelled by $\varpi$, to a new vertex. Without loss of generality one can assume that an fsa has a unique accept state, and we will make this new vertex to be the unique accept state; $s_0$ remains the start state. The language accepted by this new (1-variable) fsa is a regular language $\mathcal{R}\subset\mathrm{End}(C^*)$.
	
	For any element $\mathbf{u}\in L$, the corresponding path in $\Gamma$ is simulated in $\mathcal{R}$ by a sequence of $\phi_{\mathrm{w}}$s which produce the components of the tuple $\mathbf{u}$, followed by $\varpi$ which removes the $\perp_i$s, resulting in the element $\theta(\mathbf{u})$. Conversely, each element $h$ of the language of $\mathcal{H}$ arises from $\theta^{-1}(h)$ in $L$. So $\theta(L)$ is the language of the EDT0L system $\mathcal{H}$.
	
For (ii), to obtain $\theta_{\#}(L)$ instead of $\theta(L)$ we proceed as above, but apply the map $\perp_i\mapsto\#$ for each $i\in\{1,\ldots n-1\}$ in place of \(\varpi\). 
	\end{proof}

	By considering the case \(n=1\), it is clear that the converse of Proposition \ref{prop:forgetful} does not hold since there are EDT0L languages which are not regular, such as Example \ref{eg:AlexL}, 

\section{Equivalences between types of sets in virtually abelian groups}\label{sec:eqv_sets}

In this section we recall several important types of sets in virtually abelian groups: polyhedral, coset-wise polyhedral, rational, and semilinear, and 
establish equivalences between them. In Proposition \ref{prop:growth2} we also prove the rationality of their growth series with respect to the natural basis (and $\ell_1$ norm) of the ambient group.

It is a standard fact that we may assume that any finitely generated virtually abelian group is a finite extension of a finitely generated free abelian group. Throughout the section we let $G$ be a finitely generated virtually abelian group with free abelian normal subgroup $\Z^k$ of finite index, and use the short exact sequence 
\begin{equation}\label{ses}
	1\to\Z^k\to G\to \Delta\to 1
\end{equation}
for some finite group $\Delta$.

\subsection{Polyhedral sets}

Polyhedral sets are subsets of free abelian groups of finite rank, and play an important role in virtually abelian groups. Benson successfully translated sets of geodesic representatives in virtually abelian groups into polyhedral sets, which can be used to compute growth in effective ways.

\begin{definition}[Polyhedral sets \cite{Benson}]\label{def:PolSets}
	Let $r\in\N$, and let $\cdot$ denote the Euclidean scalar product.
	\begin{enumerate}
		\item[(i)] Any subset of $\Z^r$ of the form 
		\begin{enumerate}
		\item[(1)] $\{\mathbf{z}\in\Z^r\mid\mathbf{u}\cdot\mathbf{z}=a\}$, 
		\item[(2)] $\{\mathbf{z}\in\Z^r\mid\mathbf{u}\cdot\mathbf{z}\equiv a\mod b\}$, or
		\item[(3)] $\{\mathbf{z}\in\Z^r\mid\mathbf{u}\cdot\mathbf{z}>a\}$ 
		\end{enumerate}
		for $\mathbf{u}\in\Z^r$, $a\in\Z$, $b\in\N$, is an \emph{elementary region}, of type (1), (2), and (3) respectively;
		
		\item[(ii)]  any finite intersection of elementary regions will be called a \emph{basic polyhedral set};
		\item[(iii)]  any finite union of basic polyhedral sets will be called a \emph{polyhedral set}.
	\end{enumerate}
\end{definition}

Polyhedral sets are useful because they behave well under various set operations, as below.
 
\begin{proposition}\cite[Proposition 13.1 and Remark 13.2]{Benson}  \label{prop:polyclosed}
	Let $P, \ Q\subseteq\Z^r$ and $R\subseteq\Z^s$ be polyhedral sets for some positive integers $r$ and $s$. Then the following are also polyhedral: $P\cup Q\subseteq\Z^r$, $P\cap Q\subseteq\Z^r$, $\Z^r\setminus P$, $ P\times R\subseteq\Z^{r+s}$.
\end{proposition}

One important result, used to express sets in virtually abelian groups in terms of polyhedral sets, is the invariance of polyhedral sets under affine maps, as stated in Proposition \ref{prop:polyaffine}.
\begin{definition}
	We call a map $A\colon\Z^r\to\Z^{s}$ an \emph{integer affine map} if there exists an $s\times r$ matrix $M$ with integer entries and some constant $q\in\Z^{s}$ such that $A(p)=Mp+q$ for $p\in\Z^r$.
\end{definition}
\begin{proposition}\cite[Propositions 13.7 and 13.8]{Benson}\label{prop:polyaffine}
	Let $A$ be an integer affine map. If $ P\subseteq\Z^r$ is a polyhedral set then the image $A(P)\subseteq\Z^{s}$ is a polyhedral set. If $Q\subseteq\Z^{s}$ is a polyhedral set then the preimage $A^{-1}(Q)\subseteq\Z^r$ is a polyhedral set.
\end{proposition}

\begin{remark}\label{rem:pol}\leavevmode
	\begin{itemize}
		\item[(i)] The projection of a polyhedral set onto any subset of its coordinates is polyhedral (since such a projection is an integer affine map).
		\item[(ii)] Any subgroup of \(\Z^r\) is polyhedral.
		\item[(iii)] Elementary regions of type (1) and (2) are cosets of the subgroups $\{\mathbf{z}\in\Z^r\mid\mathbf{u}\cdot\mathbf{z}=0\}$ and $\{\mathbf{z}\in\Z^r\mid\mathbf{u}\cdot\mathbf{z}\equiv 0\mod b\}$ of $\Z^r$, respectively.
	\end{itemize}
\end{remark}

\subsection{Rational, semilinear and $N$-semilinear sets}\label{sec:ratsemilin}

We start this section by defining rational sets of elements in groups and monoids. An important feature of rational sets in groups is that the rationality of a set is independent of the choice of finite generating set, since regular languages are closed under preimages of monoid homomorphisms.

\begin{definition}
	Let $G$ be a finitely generated monoid or group, with finite generating set $S$. A subset $R\subseteq G$ is called \emph{rational} if there exists a regular language $L\subseteq S^*$ whose image in $G$ with respect to the natural projection $\varpi: S^* \mapsto G$ is the set $R$, that is, $\varpi(L)=R$.
\end{definition}

We next define semilinear and linear sets, which are widely used in computer science, and which in fact coincide with rational sets in the context of commutative monoids, as specified in Theorem \ref{thm:ES}(1).

\begin{definition}
	Let $M$ be a commutative monoid (written additively). A subset $X\subseteq M$ is \emph{linear} if there exists some finite set $B\subset M$, acting as a basis, such that $X=a+B^*$ for some $a\in M$. A subset is \emph{semilinear} if it is a finite union of linear sets.
\end{definition}

The following theorem collects key results about semilinear subsets of commutative monoids.
\begin{theorem}\label{thm:ES}
	Let \(M\) be a commutative monoid.
	\begin{enumerate}
		\item A subset \(X\subset M\) is rational if and only if it is semilinear. \cite[Section 2]{RationalSets}
		\item If $X, Y\subset M$ are rational sets, then so is $X\cap Y$. \cite[Theorem III]{RationalSets}
		\item If $M'$ is a finitely generated submonoid of a commutative monoid $M$, $X\subset M'$ a rational subset of $M$, then $X$ is a rational subset of $M'$. \cite[Corollary III.3]{RationalSets}
	\end{enumerate}
\end{theorem}

We will consider the commutative monoid $\Z^k$ generated by $\{\pm e_1,\ldots, \pm e_k\}$, where $e_i$ denotes the $i$-th standard basis vector, and establish properties of semilinear sets in $\Z^k$.

\begin{lemma}\label{lem:semilinearaffine}
	The image of any semilinear set $P \subseteq \Z^k$ under an integer affine map \(\Z^k\to\Z^k\) is semilinear.
\end{lemma}
\begin{proof}
	It is enough to prove this for an arbitrary linear set $L=a+\{b_1, b_2, \dots, b_r\}^*=a+\N b_1 + \cdots + \N b_r$. Consider the integer affine map $z\mapsto Cz+d$, where $C$ is an integer-valued $k\times k$ matrix and $d\in\Z^k$. We have
	\begin{align*}
		C(L)=C(a+\N b_1 + \cdots + \N b_r) +d = C(a)+d +\N C(b_1) + \cdots + \N C(b_r),
	\end{align*}
	which is a linear set. Since a finite union of linear sets is sent by $C$ to a finite union of images of those linear sets under $C$, we get that integer affine maps preserve semilinear sets.
\end{proof}

From Theorem IV of \cite{RationalSets}, we have the following. 
\begin{lemma} \label{lem:disjoint}
	Any semilinear subset of \(\Z^k\) can be expressed as a disjoint union of linear sets \(\bigcup_{i=1}^d(c_i+D_i^*)\) where the elements of each \(D_i\) are linearly independent.
\end{lemma}

We can now show that semilinear sets are precisely the polyhedral sets of Definition \ref{def:PolSets}.
\begin{proposition}\label{Prop:PeqS}
	A subset of \(\Z^k\) is polyhedral if and only if it is semilinear.
\end{proposition}
\begin{proof}
	First, we show that every elementary region is semilinear. Since finite unions (by definition) and finite intersections (from Theorem \ref{thm:ES} (2)) of semilinear sets are semilinear, it follows that polyhedral sets are semilinear.
	
	Elementary regions of types (1) and (2) are cosets of subgroups of $\Z^k$ (see Remark \ref{rem:pol} (iii)). Therefore for any such elementary region $E$, there exists a constant $c\in\Z^k$ (a coset representative) and a finite set of (monoid) generators $\{b_1,\ldots,b_s\}\subset\Z^k$ (with $s\leq 2k$) such that $E=c+\{b_1,\ldots,b_s\}^*$.
	
	Now suppose that $E$ is an elementary region of type (3), that is, $E=\{z\in\Z^k\mid za^{\T}\geq m\}$ for some $a\in\Z^k$ and $m\in\Z$. Since translations of semilinear sets are clearly semilinear, we assume without loss of generality that $m=0$. It is standard linear algebra to find an injective linear map $\cA\colon \Q^k\to\Q^k$ such that $\cA\colon a^{\T}\mapsto e_1^{\T}$. Let $M'\in\mathrm{GL}_k(\Q)$ be a matrix representation of $\cA$ so that we have $M'a^{\T}=e_1^{\T}$. If $M'$ contains non-integer entries, choose a positive integer $\lambda$ so that $M:=\lambda M'$ has integer entries, and therefore defines an integer affine map from $\Z^k$ to itself (where $M$ acts on the left). Now we have $Ma^{\T}=\lambda e_1^{\T}$ (and $\det M\neq0$). Consider the half-space \[E_0=\{z\in\Z^k\mid ze_1^{\T}\geq0\}=\{e_1, \pm e_2, \pm e_3,\ldots,\pm e_k\}^*,\] which is also a linear set. If $z_0\in E_0$ then $(z_0M)a^{\T} = \lambda z_0 e_1^{\T}\geq 0$, and so $z_0M\in E$, i.e. $E_0M\subset E$. (NB: we are using $M$ to define two different integer affine maps, via left action on column vectors, and right action on row vectors).
	
	Now \[E=\bigcup_{v\in P_M}(E_0+v)M\] where $P_M$ is the preimage under $M$ of the elements of the fundamental parallelepiped of the lattice defined by $M$. Since $E_0$ is semilinear, and \(P_M\) is finite (because $M$ is injective), Lemma \ref{lem:semilinearaffine} implies that $E$ is also semilinear.
	
	To show the converse, it is enough to prove that any linear set is polyhedral (since the class of polyhedral sets is closed under translation and finite union). By Lemma \ref{lem:disjoint}, and the same closure properties of polyhedral sets again, it is then enough to prove that \(D^*\) is polyhedral, whenever the elements of \(D\) are linearly independent.
	
	Linear independence implies that \(|D|\leq 2k\), and therefore we can find an integral affine map $\Z^k\to\Z^k$ which takes each of $|D|$ standard basis vectors to a unique element of $D$. The preimage of $D^*$ under this affine map will be a union of orthants, which is clearly polyhedral, and since the image of a polyhedral set is polyhedral, so is $D^*$.
\end{proof}

We next consider the growth of semilinear sets, and to do so we define monotone sets.

\begin{definition}\label{def:orthant}
	For each $I\subseteq\{1,\ldots,k\}$, including $I=\emptyset$, let \begin{equation*}
		Q_I=\{z\in\Z^k\mid e_i\cdot z\geq0\text{ if }i\in I,~e_i\cdot z<0\text{ if }i\notin I\},
	\end{equation*} i.e.~the orthant including the positive $e_i$ axis for each $i\in I$ and the negative $e_j$ axis for each $j\notin Q_I$. A subset of $\Z^k$ is said to be \emph{monotone} if it is contained entirely in a single orthant $Q_I$.
\end{definition}

\begin{lemma}\label{lem:union}
Each set $P\subseteq \Z^k$ can be written as the disjoint union $P=\bigcup_{j=1}^{2^k} P_j$ of monotone subsets $P_j$, where each $P_j$ is in a different orthant $Q_j$. Moreover, if $P$ is polyhedral then each $P_j$ is polyhedral. Equivalently, if $P$ is semilinear then each $P_j$ is semilinear, and if $P$ is rational then each \(P_j\) is rational.
\end{lemma}

\begin{proof}
We can decompose $P$ as a disjoint union of $2^k$ monotone sets as follows. Let
	$Q_1=\{\mathbf{z}\in\Z^k\mid \mathbf{z}\cdot e_i\geq0,~1\leq i\leq
	r\}=\bigcap_{i=1}^k\{\mathbf{z}\in\Z^k\mid \mathbf{z}\cdot e_i\geq 0\}$
	denote the non-negative orthant of $\Z^k$.
	Let $Q_2, \ldots, Q_{2^k}$ denote the remaining orthants (in any
	order) obtained from $Q_1$ by (compositions of) reflections along
	hyperplanes perpendicular to the axes and passing through the origin.
	 Let $ P_1=P\cap Q_1$ and for each $2\leq j\leq 2^k$, inductively define
	\[P_j = \left(P\setminus\bigcup_{i<j}P_i\right) \cap Q_j.\] Each
	$P_j$ is clearly monotone, and we
	have a disjoint union $P=\bigcup_{j=1}^{2^k} P_j$. 
	
	By construction, each orthant $Q_j$ is polyhedral and therefore if $P$ is polyhedral then every $P_j$ is also polyhedral by Proposition \ref{prop:polyclosed}.
	
	The statement about semilinearity follows from Proposition \ref{Prop:PeqS}, and the statement about rationality from Theorem \ref{thm:ES}(1).
\end{proof}

Proposition \ref{prop:growth2} is well known for polyhedral sets, but we provide here a much simpler proof than that of \cite{Benson}. Our proof relies entirely on the basic structure of semilinear sets.

\begin{proposition}\label{prop:growth2}
	Any semilinear subset (equivalently, any polyhedral subset) of \(\Z^k\) has \(\N\)-rational weighted growth series with respect to the \(\ell_1\) norm.
\end{proposition}
\begin{proof}
	As in the proof of Lemma \ref{lem:union}, we express our semilinear set as a disjoint union of monotone semilinear sets \(P_i\), each of which is in weight-preserving bijection with a semilinear set \(X_i^+\); furthermore, Lemma \ref{lem:disjoint} implies that \(X_i^+\) itself can be expressed as a disjoint union of linear sets of the form \(c_i+D_i^*\), for linearly independent sets \(D_i=\{d_{i,1},\ldots,d_{i,n}\}^*\subset \Z^k\). The weighted growth series of \(c_i+\{d_{i,1},\ldots,d_{i,n}\}^*\) will have the form
	\[z^{\|c_i\|}\prod_{j=1}^{n}\frac{1}{1-z^{\|d_{i,j}\|}},\] where \(\|\cdot\|\) denotes the weighted \(\ell_1\) norm. Since the union is disjoint, the growth series of the whole subset is simply a sum of functions of this form, and is therefore \(\N\)-rational as claimed (recalling Definition \ref{def:Nrational}).
\end{proof}
Note that this form of \(\N\)-rational function is the same as that in Proposition 14.1 of \cite{Benson}.

Finally, in his thesis, Liardet defined semilinear sets within non-commutative monoids, provided these contain some abelian submonoid. That is, even though the semilinear sets lie in a non-commutative monoid, as in Theorem \ref{Thm:Liardet}, the basis of a linear set is still required to generate an abelian monoid.

\begin{definition} \cite[Def. 4.1.2]{Liardet}\label{def:Liardet}
Let $M$ be a monoid and $N$ a submonoid of $M$. A set $X \subseteq M$ is $N$-\emph{linear} if it can be written as $X=aB^*$, where $a \in M$ and $B \subseteq N$ is a finite basis for a free abelian submonoid $B^*$ of $N$, and for which the multiplication with $a$ is bijective (that is, there is a bijection between $B^*$ and $aB^*$).

The set $X$ is \emph{$N$-semilinear} if it is the disjoint union of $N$-linear sets.
\end{definition}

 Using Liardet's definition of semilinear sets within non-commutative monoids, we can give the main theorem in Liardet's thesis. In this result the semilinear sets are in a non-commutative monoid, the virtually abelian group viewed as a monoid.

\begin{theorem} \cite[Thm. 4.1.5]{Liardet}\label{Thm:Liardet}
All rational sets of a virtually abelian group with finite index subgroup $\mathbb{Z}^k$ are $\mathbb{Z}^k$-semilinear.
\end{theorem}

\subsection{Coset-wise polyhedral sets}

The notion of a \emph{coset-wise polyhedral} (CWP) set was introduced by the second author and Levine in \cite{EvettsLevine}  as the natural set up for understanding solutions sets of equations in virtually abelian groups. 
Let $G$ be a finitely generated virtually abelian group with free abelian normal subgroup $\Z^k$ of finite index, as in the short exact sequence (\ref{ses}).

\begin{definition}\label{def:CWP}
	Let $T$ be a choice of transversal for the finite index normal subgroup $\Z^k$. A subset $V\subseteq G$ will be called \emph{coset-wise polyhedral} (CWP) if, for each $t\in T$, the set
	\[V_t =\left\{gt^{-1}\,\middle\vert\, g\in V\cap\Z^kt\right\}\subseteq\Z^{k}\]
	is polyhedral.
\end{definition}

	Note that being coset-wise polyhedral is independent of the choice of $T$. Indeed, suppose that we chose a different transversal $T'$ so that for each $t_j\in T$ we have $t'_j\in T'$ with $\Z^kt_j=\Z^kt'_j$. Then there exists $y_j\in\Z^k$ with $t_j=y_jt'_j$ for each $j$, and so $g{t'_j}^{-1}=g_jt_j^{-1}y_i$ for any $g\in\Z^kt_j=\Z^k t'_j$. So changing the transversal changes the set $V_t$ by adding a constant element $y_j$, and so it remains polyhedral by Proposition \ref{prop:polyaffine}.

It turns out that coset-wise polyhedral and rational sets coincide in virtually abelian groups, as Proposition \ref{prop:rationalCWP} shows; this relies on the following result of Grunschlag, which relates the rational subsets of a finite index subgroup of a group \(G\), to the rational subsets of \(G\) itself.

	\begin{lemma}\cite[Corollary 2.3.8]{Grunschlag_thesis}
		\label{lem:Grunschlag}
		Let \(G\) be a group with finite generating set \(S\), and \(H\) be a
		finite index subgroup of \(G\). Let \(\Sigma\) be a finite generating set
		for \(H\), and \(T\) be a right transversal for \(H\) in \(G\). For each
		rational subset \(U \subseteq G\), such that \(U \subseteq H t\) for some
		\(t \in T\), there exists a (computable) rational
		subset \(V \subseteq H\) (with respect to \(\Sigma\)), such that \(U = Vt\).
	\end{lemma}

\begin{proposition}\label{prop:rationalCWP}
	Let $U$ be a subset of a virtually abelian group $G$. Then $U$ is a rational set if and only if it is coset-wise polyhedral.
\end{proposition}
\begin{proof}
	As above, we work with the generating set $\{\pm e_1,\ldots,\pm e_k\}\cup T$ for $G$, where $\{\pm e_1,\ldots,\pm e_k\}$ generates $\Z^k$, and \(T=\{t_1,\ldots,t_d\}\) is a transversal. Firstly, suppose $U$ is rational. Then Lemma \ref{lem:Grunschlag} gives a finite set of rational subsets of $\Z^k$, say $V_1,\ldots,V_d$, such that $U=\bigcup_i V_it_i$. Since rational subsets of $\Z^k$ are polyhedral by Theorem \ref{thm:ES}(1) and Proposition \ref{Prop:PeqS}, $U$ is coset-wise polyhedral.
	
	Conversely, suppose that $U\subset G$ is coset-wise polyhedral. For each $t_i\in T$ the set $U_i:=\{ut_i^{-1}\mid u\in U\cap\Z^kt_i\}\subset\Z^k$ is polyhedral, and hence rational (by Proposition \ref{Prop:PeqS} and Theorem \ref{thm:ES}(1). So there exists a regular language $L_i\subset\{\pm e_1,\ldots,\pm e_k\}^*$ which surjects to $U_i$. Then regular language $\bigcup L_it_i\subset\left(\{\pm e_1,\ldots,\pm e_k\}\cup T\right)^*$ surjects to $U$, which is thus rational.
\end{proof}

When dealing with sets of tuples of elements of a virtually abelian group $G$, we can think of them as subsets of the direct product of finitely many copies of $G$. The following Lemma shows that any results about subsets of $G$ will also hold for sets of tuples.
\begin{lemma} \label{lem:directprod}
	If $G$ is virtually abelian, with index-$d$ normal subgroup $\Z^k$ and transversal $T$ as usual, then the direct product $G^n$ is virtually $\Z^{kn}$, with transversal of size \(dn\), given by the set of products $T^n=\{t_{i_1}t_{i_2}\cdots t_{i_n}\in G^n\mid t_{i_k}\in T\}$, with the \(k\)th term \(t_{i_k}\) coming from the \(k\)th factor of \(T\) in \(G^n\).
\end{lemma}
\begin{proof}
	Any element of $G^n$ can be put into the following form:
	\[x_1t_{i_1}\cdot x_2t_{i_2}\cdots x_nt_{i_n} = x_1x_2\cdots x_n \cdot t_{i_1}t_{i_2}\cdots t_{i_n}\in\Z^{kn}t_{i_1}t_{i_2}\cdots t_{i_n}.\]
	Since $T^n$ has exactly $dn$ elements, we have $[G^n\colon\Z^{kn}]= dn$.
\end{proof}

\section{Definable sets in virtually abelian groups}

We give here an overview of definable sets in virtually abelian groups, and refer the reader to the books \cite{Prest} and \cite{TentZiegler} for an in-depth account.

Let $n$ be an integer with $n \geq 1$. A subset $S$ of a group $G$ is $n$-\emph{definable} over $G$ if there is a first-order formula $\Phi(x, y_1, \dots, y_n)$ and an $n$-tuple of parameters $\overline{m}=(m_1, \dots, m_n)$, $m_i \in G$ with $1\leq i \leq n$, such that 
$\Phi(g,\overline{m})$ is true if and only if $g\in S$. We say that a set is $0$-definable if there are no parameters $\overline{m}$ in the formula $\Phi$, and say that it is \emph{definable} if it is $n$-definable for some $n\geq 1$ as above.
More generally, if the formula $\Phi$ is over a tuple $\overline{x}=(x_1, \dots, x_r)$ of $r$ variables instead of a single one, that is, it has the form $\Phi(\overline{x}, y_1, \dots, y_n)$, then we have definable sets $S \subseteq G^n$ consisting of tuples of elements in $G$. If for a group $G$ a subgroup $H$ of the direct product $G^n$ is definable over $G$, then $H$ is called a \emph{definable subgroup}. 

\newpage

\begin{example}\leavevmode
\begin{itemize}
\item[(1.)] An algebraic set over a group $G$ is the solution set to a system of equations with coefficients in $G$. For a single equation $E(\overline{x}, \overline{m})=1$ with coefficients $\overline{m}$ in $G$, an algebraic set is given by an existential formula of the form $\exists \overline{x} E(\overline{x}, \overline{m})=1$. For example, an equation $x_1^2x_2^2h=1$, where $h\in G$, gives rise to an algebraic set consisting of all pairs $(g_1, g_2)\in G^2$, where $g_1, g_2$ are such that $g_1^2g_2^2h=1$ in $G$.
\item[(2.)] The conjugacy class of any element $m \in G$ is a definable set, with formula $\Phi(x,m)$ given by $\exists y, x=y^{-1}my$.
\item[(3.)] The set of elements of order $2$ in a group $G$ is definable and satisfies the formula $\Phi(x)$ (with no coefficients) given by $x^2=1 \wedge x \neq 1 $.
\end{itemize}
\end{example}

Virtually abelian groups have been studied in model theory primarily in a module setup. Let $G$ be a finitely generated virtually abelian group with free abelian normal subgroup $A=\Z^k$ of finite index, as in (\ref{ses}), and finite quotient $\Delta=G/A$. Then from the sequence 
\begin{equation*}
	1\to A\to G\to \Delta\to 1,
\end{equation*}
we can view $A$ as a right $\Z[\Delta]$-module, where the action of the group algebra $\mathbb{Z}[\Delta]$ on $A$ extends the conjugation action of $\Delta$ on $A$. That is, if $a\in A$ and $g \in G$ with $\bar{g}=gA \in \Delta$ representing the $g$-coset of $A$ in $\Delta$, then $\bar{g}\circ a=g^{-1}ag$.
 Then for a module, such as $A$, over a ring $R$ (such as $\Z[\Delta]$), every formula $\Phi(\overline{x}, y_1, \dots, y_n)$ is equivalent to a Boolean combination of \emph{positive primitive} formulas (see \cite[Theorem 3.3.5]{TentZiegler}), defined as follows.
 
 \begin{definition}\leavevmode
 \begin{itemize}
 \item[(1)] An \emph{equation} over $\overline{z}=(z_1, \dots, z_n)$ in an $R$-module $A$ is a formula $\Psi(\overline{z})$ of the form \[r_1 \circ z_1 + \dots r_n \circ z_n = 0,\] where $r_i \in R$.
  \item[(2)] A \emph{positive primitive} (pp) formula is of the form \[ \exists \ \overline{y} \ \Psi_1(\overline{x}, \overline{y})\wedge \dots \wedge \Psi_k(\overline{x}, \overline{y}),\] where all $\Psi_j(\overline{z})$ are equations over $\overline{z}=(\overline{x}, \overline{y})$ (and so $\Psi_1(\overline{x}, \overline{y})\wedge \dots \wedge \Psi_k(\overline{x}, \overline{y})$ is a system of equations).
 \end{itemize}
 \end{definition}
 
 That is, definable sets in virtually abelian groups are obtained from projecting solution sets of systems of equations onto their first coordinates ($|\overline{x}|$ - many coordinates according to our definition). It is easy to see that pp formulae define subgroups and cosets of subgroups of $G^n$ (see \cite[Lemma 3.3.7]{TentZiegler}). In fact, all definable sets in virtually abelian groups have such a characterisation, by the work of Hrushovski and Pillay:

\begin{theorem}\cite[Theorems 4.1 and 3.2]{HrushovskiPillay}\label{thm:defVirtAb}
Given a virtually abelian group $G$, every definable set $X \subseteq G^n$ is a Boolean combination of cosets of definable subgroups of $G^n$, for any $n\geq 1$.
\end{theorem}

\begin{corollary}\label{Cor:DefRat}
	Given a virtually abelian group $G$, every definable set $X \subseteq G^n$ is a rational subset of $G^n$, for any $n\geq 1$.
\end{corollary}
\begin{proof}
	We work within the direct product $G^n$, which by Lemma \ref{lem:directprod} is also virtually abelian. Subgroups of finitely generated virtually abelian groups are finitely generated, and are therefore rational subsets. Furthermore, any coset of a finitely generated subgroup is rational since it is just a translation of a rational set. The result follows from Theorem \ref{thm:defVirtAb} since rational sets are closed under Boolean combinations.
\end{proof}

\section{Natural normal forms}\label{sec:NF}

Write $\Sigma=\{a_1,A_1,\ldots,a_k,A_k\}$ for the set of positive and negative standard generators of $\Z^k$, that is, $A_i=a_i^{-1}$. For $g\in\Z^k$, set $\mathrm{NF}(g)$ to be the shortlex representative of $g$ with respect to $\Sigma$ with the order given above. The set of all such representatives is a regular language, denoted as follows:
\[\mathrm{NF}(\Z^k)=(a_1^*\cup A_1^*)(a_2^*\cup A_2^*)\cdots(a_k^*\cup A_k^*)\subset \Sigma^*,\]
where we write $a_i^*$ instead of $\{a_i\}^*$ for simplicity.

For a finitely generated virtually abelian group $G$, with finite index normal subgroup $\Z^k$, we use $\mathrm{NF}(\Z^k)$, together with some choice of transversal $T=\{t_1,\ldots,t_d\}$ for the cosets of $\Z^k$ to describe a normal form for $G$:
\[\mathrm{NF}(G)=(a_1^*\cup A_1^*)(a_2^*\cup A_2^*)\cdots(a_k^*\cup A_k^*)(t_1 \cup t_2 \cup \cdots \cup t_d)=\mathrm{NF}(\Z^k)T.\] Furthermore, by Lemma \ref{lem:directprod}, the direct product \(G^n\) is virtually abelian, and we define the natural normal form
\[\mathrm{NF}(G^n)=\mathrm{NF}(G)^n T^n.\] The normal form $\mathrm{NF}(G)$ will not be geodesic, in general.

We can also see \(\mathrm{NF}(\Z^k)\) as a set of tuples, and then get a \(k\)-regular ($k$-variable) language:
\[\mathrm{NF}_k(\Z^k)=\left((a_1^*\cup A_1^*),(a_2^*\cup A_2^*),\cdots, (a_k^*\cup A_k^*)\right)\subset(\Sigma^*)^k.\]
We then have \(\theta(\mathrm{NF}_k(\Z^k))=\mathrm{NF}(\Z^k)\) where \(\theta\) is the `forgetful' morphism of Proposition \ref{prop:forgetful}.

In this section we will study the formal language properties of $$\mathrm{NF}(U)=\{w\in\mathrm{NF}(G)\mid \varpi(w)\in U\}$$ (for the natural projection \(\varpi\colon \Sigma^*\to G\)) when $U$ is a rational subset of a virtually abelian group $G$. Firstly, we show that semilinear sets which are monotone (see Definition \ref{def:orthant}) have \(k\)-regular normal forms.
\begin{proposition}\label{prop:monotonenregular}
	If $X\subset\Z^k$ is monotone and semilinear, then $\mathrm{NF}_k(X)$ is $k$-regular.
\end{proposition}
\begin{proof}
	By Lemma \ref{lem:regularunion}, we may assume without loss of generality that \(X\) is linear, say of the form \(c+\{d_1,d_2,\ldots,d_r\}^*\). So each element of \(X\) has the form \(c+ m_1d_1+m_2d_2+\cdots +m_rd_r\) for some \(m_j\in\N\). For any \(b\in\Z^k\), let \(p_b\) denote a path of \(|b|_{\ell_1}\) consecutive edges (with \(|b|_{\ell_1}-1\) states), with each edge labelled by a \(k\)-tuple in \(\left( \{a_1, A_1, \varepsilon \},\{a_2, A_2, \varepsilon\},\ldots,\{a_k, A_k, \varepsilon\}\right)\) with exactly one non-epsilon entry, so that the \(k\)-tuple obtained by reading along the path and deleting \(\varepsilon\)s is equal to \(b\).
	\begin{figure}[!h]
	\includegraphics[width=0.6\textwidth]{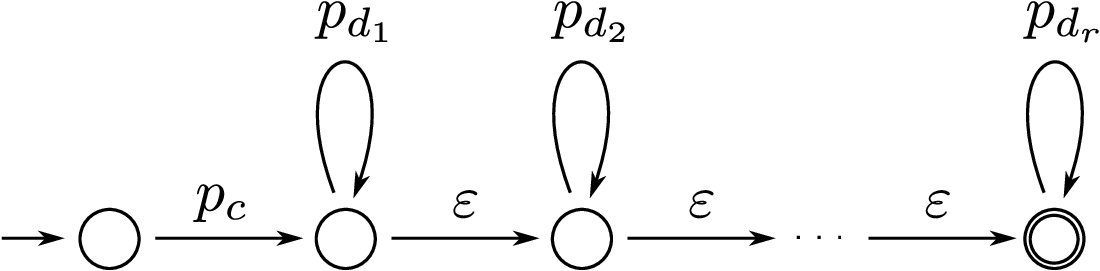}
	\caption{The \(k\)-fsa for Proposition \ref{prop:monotonenregular}.}
	\label{fig:monotoneFSA}
\end{figure}

	By monotonicity, each \(d_j\) lies in the same orthant, and therefore a concatenation of paths \(p_{d_j}p_{d_{j+1}}\) can never result in a subword \(a_iA_i\) or \(A_ia_i\), and therefore produces an element of \(\mathrm{NF}(G)\). The automaton given in Figure \ref{fig:monotoneFSA} then clearly produces the language \(\mathrm{NF}_k(X)\).
\end{proof}

We are now able to put together Propositions \ref{prop:forgetful} and  \ref{prop:monotonenregular} to set up a bridge from virtually abelian groups to semilinear sets in free abelian groups, and from $k$-regular languages to EDT0L. We thus show that any rational set in a finitely generated virtually abelian group has an EDT0L representation in terms of the most natural normal forms for the group.

\begin{theorem}\label{thm:NF}
	Let $U$ be a rational subset of a virtually abelian group. Then $\mathrm{NF}(U)$ is an EDT0L language.
\end{theorem}
\begin{proof}
	By Proposition \ref{prop:rationalCWP}, we have a disjoint union \[U=\bigcup_{i=1}^d U_it_i\] where each $U_i=\{ut_i^{-1}\mid u\in U\cap\Z^k\}\subset\Z^k$ is a polyhedral subset. We now claim that polyhedral sets have $k$-regular normal forms. More precisely, \(\mathrm{NF}_k(P)\) is \(k\)-regular whenever \(P\subset\Z^k\) is a polyhedral set. The result then follows from the fact that \(k\)-regular languages are EDT0L (Proposition \ref{prop:forgetful}), and that the class of EDT0L languages is closed under concatenation with a single letter \(t_i\), and under finite unions (Lemma \ref{lem:EDT0Lclosure}).

	To prove the claim, consider a polyhedral set $P\subset \Z^k$. We can decompose $P$ as a disjoint union of $2^k$ monotone polyhedral sets as in Lemma \ref{lem:union}. Then by Proposition \ref{prop:monotonenregular}, $\mathrm{NF}_k(P)$ is a finite union of $k$-regular languages, and is therefore itself $k$-regular.
\end{proof}
Finally, since definable sets in virtually abelian groups are rational sets in direct products of virtually abelian groups, and these products are again virtually abelian (Lemma \ref{lem:directprod}), we are able to use the rationality of definable sets and the representation of rational sets in terms of EDT0L languages to get the following.
\begin{corollary}
	Let $G$ be a finitely generated virtually abelian group, and $U\subset G^n$ a definable set. Then $\mathrm{NF}(U)\subset\mathrm{NF}(G^n)$ is EDT0L.
\end{corollary}

Since \(\mathrm{NF}(G)\) is \emph{not} a geodesic normal form, it cannot be used to compute the growth series of $U$ with respect to the word metric on the group. For that we will need a different normal form, which we study in the next section.

\section{Geodesic normal forms}\label{sec:gnf}

In \cite{NeumannShapiro97}, Neumann and Shapiro show that there is a virtually abelian group $G$ with fixed generating set $X$ such that no regular (i.e. $1$-regular) language of geodesics can surject to the elements of $G$. So in particular, there are virtually abelian groups with generating sets for which \emph{no} geodesic normal form is regular. This does not imply that there is a virtually abelian group $G$ for which one cannot obtain regular geodesic normal forms irrespective of generating set: one can always get such a normal form after possibly enlarging the generating set that was given, as shown in a different paper by Neumann and Shapiro, (\cite{NeumannShapiro95}).

Here we complete the picture to show that for \emph{any} monoid generating set $\Sigma$ of a virtually abelian group \(G\) there always exists some geodesic normal form $\GF(G, \Sigma)$, which is the image under \(\theta\) (see Definition \ref{def:psipi}) of an \(m\)-regular language for an appropriate value to \(m\). We will use $\GF(G)$ instead of $\GF(G, \Sigma)$ in most cases since $\Sigma$ should be clear from the context (see Notation \ref{not:GF}).

 Furthermore, we show that there is a subset of $\GF(G)$ consisting of geodesic representatives for the conjugacy classes of the group; moreover, for any subgroup, there is a subset of $\GF(G)$ consisting of geodesic representatives for the cosets. For any rational subset of the group, the normal form representatives (contained in $\GF(G)$) are also the image of an \(m\)-regular language. In particular, the $\GF(G)$-representatives of definable sets form EDT0L languages; that is, the same result as for the normal form $\NF(G)$ in the previous section holds for $\GF(G)$.

Benson introduced a normal form for elements of virtually abelian groups in 1983 \cite{Benson}. Unlike $\mathrm{NF}(G)$, it consists of geodesic representatives. However, the construction is much more involved. We give a brief overview below.

Let $G$ be virtually abelian with index-$d$ subgroup $\mathbb{Z}^k$, and choose a finite generating set $\Sigma$ that generates \(G\) as a monoid. As in Section \ref{sec:prelims}, for a word \(w\in\Sigma^*\), we will write \(\overline{w}\) for the element it represents in \(G\). Similarly, the image of a subset \(W\subset\Sigma^*\) in \(G\) will be denoted \(\overline{W}\).
A function $\|\cdot\|\colon \Sigma\rightarrow\N_+$ will be called a \emph{weight function}. We extend this to $\|\cdot\|\colon \Sigma^*\rightarrow\N$, so that $\|s_1s_2\cdots s_l\|=\|s_1\|+\|s_2\|+\cdots+\|s_l\|$ for any word $s_1s_2\cdots s_l$. Define the weight of a group element as \[\|g\|=\mathrm{min}\left\{\|w\|\mid w\in \Sigma^*,~\overline{w}=g\right\}.\] If $\|s\|=1$ for all $s\in \Sigma$, this gives the usual notion of word length.

\begin{definition}
	With $\Sigma$ and $d$ as above, define an extended generating set
	\begin{equation}\label{eq:Stilde}
		S=S_\Sigma=\{s_1s_2\cdots s_k\mid s_i\in \Sigma,~1\leq k\leq d\}\subset G.
	\end{equation}
	The weight of a generator $s_1s_2\cdots s_k\in S$ is defined with respect to our original weight function: $\|s_1s_2\cdots s_k\|=\sum_i\|s_i\|$.
\end{definition}
\begin{remark}\label{rem:gen_sets}
The weight of an element with respect to the new weighted generating set $S$ is equal to its weight with respect to $\Sigma$, and so the respective weighted growth series are equal.

Moreover, any geodesic word (i.e. a weight-minimal representative for an element) over $S=S_{\Sigma}$ is geodesic over the initial generating set $\Sigma$.
\end{remark}

\begin{definition}\label{def:XY}
	Write $X:=S\cap\Z^k=\{x_1,\ldots,x_r\}$ and $Y:=S\setminus(S\cap\Z^k)=\{y_1,\ldots,y_s\}$, and call any word in $Y^*$ a \emph{pattern}.
\end{definition}

\begin{definition}
	Let $p\colon S\rightarrow Y$ be the map that records the generators not in $\Z^k$: \[p\colon s_i\mapsto\begin{cases} \varepsilon & \text{ if }s_i\in X\\ s_i & \text{ if }s_i\in Y \end{cases}.\] This extends to a monoid homomorphism $p\colon S^*\rightarrow Y^*$, which ignores those generators in a word that belong to $\Z^k$. For $w\in S^*$ we call $p(w)$ the \emph{pattern} of $w$.
\end{definition}
\begin{definition}\label{def:Wpi}
	Let $\pi=y_1y_2\cdots y_l\in Y^*$ be some pattern. Then we denote by $W^\pi$ the subset of $S^*$ consisting of \emph{$\pi$-patterned words}, that is, all words of the form
	\begin{equation}\label{eq:Wpi}
		x_1^{i_1}x_2^{i_2}\cdots x_r^{i_r} y_1 x_1^{i_{r+1}}x_2^{i_{r+2}}\cdots x_r^{i_{2r}} y_2 x_1^{i_{2r+1}}x_2^{i_{2r+2}}\cdots x_r^{i_{3r}} \cdots y_l x_1^{i_{lr+1}}x_2^{i_{lr+2}}\cdots x_r^{i_{lr+r}}
	\end{equation}
	where $i_j\in\N$.
\end{definition}
Note that every group element has a geodesic representative in some $W^{\pi}$ (since reordering the powers of \(x_1,\ldots,x_r\) does not change the element represented, and cannot increase the length of the word). This definition allows us to identify patterned words with vectors of non-negative integers, by focussing on just the powers of the generators in $X$ as follows.
\begin{definition}\label{def:phipi}
	Fix a pattern $\pi$ of length $l$, and write $m_\pi=lr+r$, where $r=|X|$ as above. Define the bijection $\phi_{\pi}\colon W^{\pi}\to\N^{m_\pi}$ that records, for every patterned word, the exponents of its $\mathbb{Z}^k$ coordinates:
	\[\phi_\pi\colon x_1^{i_1}x_2^{i_2}\cdots x_r^{i_r} y_1 x_1^{i_{r+1}}x_2^{i_{r+2}}\cdots x_r^{i_{2r}} y_2 \cdots y_l x_1^{i_{lr+1}}x_2^{i_{lr+2}}\cdots x_r^{i_{lr+r}}\mapsto (i_1, \ i_2, \ \ldots, \ i_{lr+r}).\]
\end{definition}
It will also be useful to view elements of $W^\pi$ as tuples of powers of generators, as follows.
\begin{definition}\label{def:psipi}
	Fix a pattern $\pi$. Let $\psi_\pi\colon W^\pi\to (S^*)^{m_\pi+|\pi|}$ be the bijection between words and tuples given by
	\begin{align*}
		\psi_\pi\colon& x_1^{i_1}x_2^{i_2}\cdots x_r^{i_r} 	y_1 x_1^{i_{r+1}}x_2^{i_{r+2}}\cdots x_r^{i_{2r}} y_2 \cdots y_l x_1^{i_{lr+1}}x_2^{i_{lr+2}}\cdots x_r^{i_{lr+r}}\\
		&\mapsto (x_1^{i_1}, x_2^{i_2},\ldots,x_r^{i_r}, y_1,x_1^{i_{r+1}}, x_2^{i_{r+2}},\ldots,x_r^{i_{2r}},\ldots, y_l, x_1^{i_{lr}}, x_2^{i_{lr+2}},\ldots, x_r^{i_{lr+r}}).
	\end{align*}
	Note that \(\psi_\pi\) is precisely the inverse of the map \(\theta\) defined in Proposition \ref{prop:forgetful}.
\end{definition}

Definition \ref{def:phipi} will allow us to work with subsets of $\Z^{m_\pi}$ in place of sets of words. We apply the weight function $\|\cdot\|$ to $\Z^{m_\pi}$ in the natural way, weighting each coordinate with the weight of the corresponding $x\in X$. More formally, we have
\begin{equation*}
	\|(i_1, \ \ldots, \ i_{m_\pi})\| := \sum_{j=1}^{m_\pi} i_j\|x_{j\bmod r}\|,
\end{equation*}
where we take \(\|x_0\|=\|x_r\|\).
Then $\phi_\pi$ preserves the weight of words in $W^\pi$, up to a constant:
\begin{equation*}
	\|\phi_\pi(w)\| = \|w\| - \|\pi\|.
\end{equation*}
Fix a transversal $T$ for the cosets of $\Z^k$ in $G$. Note that, since $\Z^k$ is a normal subgroup, we can move each $y_i$ in the word \eqref{eq:Wpi} to the right, modifying only the generators from $X$, and we have $\overline{w}\in\Z^k\overline{\pi}$ for any $w\in S^*$ as in (\ref{eq:Wpi}) and $\overline{\pi}\in T$. Thus $\overline{W^\pi}\subset\Z^kt_\pi$ for some $t_\pi\in T$ where $\overline{\pi}\in\Z^kt_\pi$, i.e. the $\Z^k$-coset of $\overline{w}$ depends only on its pattern $p(w)$.

It turns out that we can pass from a word $w\in W^\pi$ to the normal form $\mathrm{NF}(\overline{w})$ of Section \ref{sec:NF} (with respect to $T$ and the standard basis for $\Z^k$) using an integral affine map.

\begin{proposition}[Section 12 of \cite{Benson}]\label{prop:affA}
	For each pattern $\pi\in Y^*$, there exists an integer affine map $\cA_\pi\colon\N^{m_\pi}\to\Z^k$ such that for each $w\in W^\pi$, $\overline{w}=\left(\cA_\pi\circ \phi_\pi(w)\right)t_\pi\in\mathrm{NF}(G)$.
\end{proposition}

One of the key ingredients in Benson's work is to consider a certain finite set of patterns, more specifically, the ones of length at most the index of the maximal normal free abelian subgroup in $G$, as below. As the results in the remaining of the section show, it is enough to consider this finite set of patterns to obtain representatives of elements, conjugacy classes, cosets, etc for the entire group.

\begin{definition}
	Let $P\subset Y^*$ be the set of all patterns of length at most $d=[G\colon\Z^k]$.
\end{definition}

The following result provides a set of geodesic representatives for the cosets of a fixed subgroup $H$ by picking out the coset representatives belonging to each $W^{\pi}$, when $\pi$ is one of the finitely many patterns in $P$.

\begin{theorem}[\cite{Benson}, \cite{Evetts}]\label{thm:cosetreps}
	Let $G$ be a virtually abelian group with generating set $\Sigma$ and extended generating set $S=S_{\Sigma}$. Let $H$ be any subgroup of $G$. For each $\pi\in P$, there exists a set of words $U_H^\pi\subset W^\pi \subset S^*$, such that
	\begin{enumerate}
		\item the disjoint union $\bigcup_{\pi\in P} U^\pi_H$ consists of exactly one geodesic representative for every coset in $G/H$, and
		\item for each $\pi$, $\phi_\pi(U_H^\pi)\subset\Z^{m_\pi}$ is a polyhedral set.
	\end{enumerate}
\end{theorem}
\begin{notation}\label{not:GF}\leavevmode
\begin{enumerate}
\item We use the notation \[\GF(G/H)=\bigcup_{\pi\in P} U^\pi_H\] to capture the set of geodesic representatives for the cosets of $H$ in $G$ in Theorem \ref{thm:cosetreps}(1). When $H=\{1\}$, Theorem \ref{thm:cosetreps} gives a geodesic normal form for the elements of $G$ with respect to the extended generating set $S_\Sigma$ (and therefore with respect to $\Sigma$, see Remark \ref{rem:gen_sets}), so \[\GF(G)=\GF(G/\{1\})\] is the geodesic normal form mentioned at the beginning of the section, to be distinguished from the natural normal form $\NF(G)$ of the previous section.

\item In general, if $V\subseteq G$ is a subset of $G$ or if $V$ is a set of objects with representatives in $G$ (such as cosets, or conjugacy classes), we denote by $\GF(V) \subseteq \GF(G)$ the set of geodesic representatives for elements in $V$, with one representative per element.
\end{enumerate}
\end{notation}

The following theorem about conjugacy representatives follows from \cite{Evetts} but isn't explicitly stated there. Note that in that paper, the sets $U_c^\pi$ are called $\mathcal{L}_\pi$.
\begin{theorem}[\cite{Evetts}]\label{thm:conjreps} Let $G$ be a virtually abelian group and let $\mathcal{C}=\mathcal{C}(G)$ be the set of conjugacy classes of $G$.

	For each $\pi\in P$ there exists a set of words $U_c^\pi\subset W^\pi$ such that
	\begin{enumerate}
		\item The disjoint union $\bigcup_{\pi\in P}U_c^\pi$ consists of exactly one geodesic representative for every conjugacy class of $G$, that is, 
		\[\GF(\mathcal{C})=\bigcup_{\pi\in P}U_c^\pi,\]
		
		\item and for each $\pi$, $\phi_\pi(U_c^\pi)\subset\Z^{m_\pi}$ is a polyhedral set.
	\end{enumerate}
\end{theorem}
 
 Lemma \ref{lem:semilinearregular} translates (patterned) words over the extended generating set $S$ of $G$ into the corresponding tuples, and thus connects semilinear sets and $n$-regular $n$-variable sets.

\begin{lemma}\label{lem:semilinearregular}
	If $V\subset W^{\pi} \subset S^*$ is a set of $\pi$-patterned words, and $\phi_{\pi}(V) \subset \N^{m_{\pi}}$ is semilinear, then the set of tuples $\psi_\pi(V)\subset(S^*)^{m_\pi+|\pi|}$ corresponding to $V$ is $(m_\pi+|\pi|)$-regular.
\end{lemma}
\begin{proof}
	By hypothesis, $\phi_\pi(V)\subset\N^{m_\pi}$ is a union of linear subsets of $\N^{m_\pi}$. Since a union of $n$-regular languages is $n$-regular (Lemma \ref{lem:regularunion}), it suffices to prove the Lemma for the case where $\phi_\pi(V)$ is linear, say of the form $a+\{b_1,\ldots,b_k\}^*$ for $a,b_1,\ldots,b_k\in\N^{m_\pi}$. We construct an $(m_\pi+|\pi|)$-variable finite state automaton that accepts $\psi_\pi(V)$.
	
	Given a vector $\bm{i}=(i_1,i_2,\ldots,i_{m_{\pi}})\in\N^{m_\pi}$, denote by $p_{\bm{i}}$ a path of $|\bm{i}|_{\ell_1}$ consecutive edges, each labelled by an $(m_{\pi}+|\pi|)$-tuple with an element of $S$ in one component and $\varepsilon$'s elsewhere, such that the tuple obtained by reading along the path and then deleting $\varepsilon$'s is equal to $(x_1^{i_1},x_2^{i_2},\ldots,x_r^{i_r},\varepsilon,x_1^{i_{r+1}},x_2^{i_{r+2}},\ldots,x_r^{i_{2r}},\varepsilon,\ldots,\varepsilon,x_1^{i_{kr+1}},x_2^{i_{kr+2}},\ldots,x_r^{i_{kr+r}})$, i.e. the tuple $\psi_\pi\circ\phi_\pi^{-1}(\bm{i})$ with the $y_i$s removed.
	
	Our automaton has a single start state $s_s$, a single accept state $s_a$, and $k$ marked states, labelled $s_1,\ldots,s_k$. Between $s_s$ and $s_1$, there is a path $p_a$ (with extra states added as necessary). For each $i\in\{1,\ldots,k\}$, there is a path $p_{b_i}$ that starts and ends at state $s_k$ (with extra states added as necessary). Between $s_i$ and $s_{i+1}$ (for each $i\in\{1,\ldots,k-1\}$), there is a single edge labelled with a vector of epsilons. Finally, from $s_k$ to $s_a$ there is a path $p_\pi$. See figure \ref{fig:semilinearregular} for a schematic diagram.
\end{proof}
\begin{figure}[!h]
	\centering
	\includegraphics[width=0.7\textwidth]{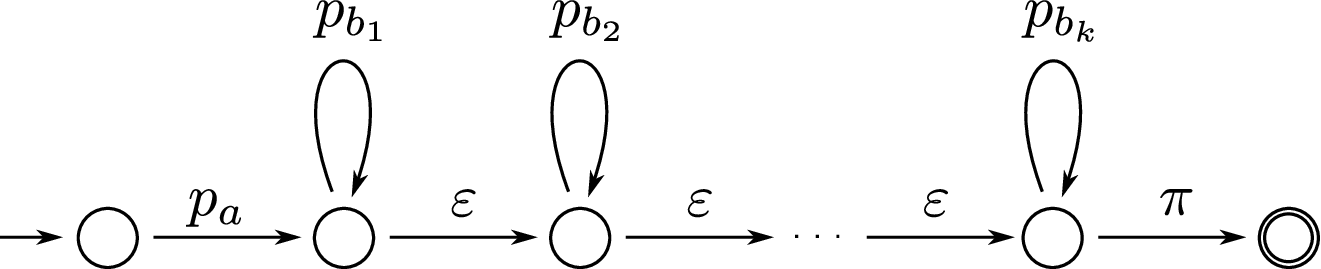}
	\caption{The \((m_\pi + |\pi|)\)-fsa of Lemma \ref{lem:semilinearregular}.}
	\label{fig:semilinearregular}
\end{figure}

Now we can prove the existence of a geodesic normal form that is $m$-regular.
\begin{theorem}\label{thm:GF}
	The geodesic normal form \(\mathrm{GF}(G)\) has the property that \(\psi_\pi(\mathrm{GF}(G))\) is an \(m\)-regular language for some $m$ depending on the generating set (where $\psi_{\pi}$ is the map from Definition \ref{def:psipi}). 
	
	Furthermore, the geodesic normal form representatives for cosets of any subgroup, and for conjugacy classes, also form $m$-regular languages (in their image under \(\psi_\pi\)).
\end{theorem}
\begin{proof}	
	From Theorem \ref{thm:cosetreps}, the set \(\mathrm{GF}(G)\) is a union of sets of the form \(U^{\pi}\), for some pattern \(\pi\) in $P$, each with the property that \(\phi_{\pi}(U^\pi)\subset\Z^{m_\pi}\) is polyhedral, and hence semilinear. Since each \(\phi_\pi(U^\pi)\) is contained in \(\N^{m_\pi}\), Theorem \ref{thm:ES} (3) implies that they are in fact semilinear subsets of $\N^{m_\pi}$. Lemmas \ref{lem:semilinearregular} and \ref{lem:regularunion} then finish the argument.
	
	An analogous argument applies for the sets \(\mathrm{GF}(G/H)\) and for \(\mathrm{GF}(\mathcal{C})\), via Theorem \ref{thm:conjreps}.
\end{proof}

It follows as a corollary of Theorems \ref{thm:cosetreps} and \ref{thm:conjreps} that the weighted standard, coset, and conjugacy growth series of $G$ are rational functions. See \cite{Evetts} for full details. We now complete the picture by demonstrating that any rational subset of $G$ also has rational weighted growth series and is $m$-regular. The following is a special case of Theorem 4.15 of \cite{EvettsLevine}. We include the proof here for completeness.

\begin{theorem}\label{thm:CWPrational}
Let G be a finitely generated virtually abelian group.
	If $V\subset G$ is coset-wise polyhedral then it has \(\N\)-rational (relative) weighted growth series.
\end{theorem}
\begin{proof}
	Fix a transversal $T$. For each $t\in T$, let $P_t\subset P$ denote the set of patterns $\pi$ with $\overline{\pi}\in\Z^kt$. The following subset of $S^*$ is a set of unique geodesics representatives for the elements of $V$.
	\begin{align*}
		\mathcal{V}:=\bigcup_{t\in T}\bigcup_{\pi\in P_t}\left\{u\in U^\pi\mid \cA_\pi\circ\phi_\pi(u)\in V_t\right\} = \bigcup_{t\in T}\bigcup_{\pi\in P_t} U^\pi \cap \left(\phi_\pi^{-1}\circ\cA_\pi^{-1}(V_t)\right).
	\end{align*}
	Applying $\phi_\pi$ to a component of this union yields a set $\phi_\pi(U^\pi)\cap \cA_\pi^{-1}(V_t)$, which is polyhedral (since $V_t$ is polyhedral), and therefore has \(\N\)-rational weighted growth series (by Proposition \ref{prop:growth2}). Since the bijection $\phi_\pi$ is length-preserving (up to addition of a constant), each component of the union has \(\N\)-rational weighted growth series, and therefore so does $V$ itself.
\end{proof}

In light of Proposition \ref{prop:rationalCWP}, which establishes the equivalence of coset-wise polyhedral and rational sets, Theorem \ref{thm:CWPrational} gives the following.
\begin{corollary} \label{cor:rational-rational}
	Let $G$ be a virtually abelian group. Rational subsets have \(\N\)-rational (relative) weighted growth series with respect to any generating set of $G$.
	 The following types of sets are rational, and therefore have rational growth series with respect to any set of generators of $G$:
\begin{enumerate}
\item elements of any fixed subgroup,
\item coset representatives of a fixed subgroup,
\item algebraic sets,
\item definable sets,
\item conjugacy representatives.
\end{enumerate}
\end{corollary}
We can easily generalise Theorem \ref{thm:GF} and obtain geodesic normal form representatives for coset-wise polyhedral sets, and therefore rational sets in general.
\begin{theorem}
	Let $G$ be a virtually abelian group and $V$ a coset-wise polyhedral subset of $G$. The geodesic normal form representatives given by $\GF(V)$ form an $m$-regular language when viewed as tuples for some positive integer $m$.
\end{theorem}
\begin{proof}
	Let $V$ be coset-wise polyhedral. As in the above proof, the Benson normal form (GF) representation of $V$ is a finite union of sets of the form $\phi_\pi(U^\pi)\cap \cA_\pi^{-1}(V_t)$, which are polyhedral and hence semilinear. The result then follows from Lemma \ref{lem:semilinearregular} and Lemma \ref{lem:regularunion}.
\end{proof}
Finally, Proposition \ref{prop:forgetful}, which shows how the forgetful map takes $n$-regular $n$-variable language to EDT0L ones, implies that any rational set has unique geodesic representatives that form an EDT0L language.
\begin{corollary}\label{cor:NFrational}
	Let $G$ be a virtually abelian group with finite monoid generating set $\Sigma$, and let $R \subseteq G$ a rational set. 
	There exists some positive integer $m$ such that the geodesic normal form representatives $\GF(R)$, written as $m$-tuples, form an $m$-variable language that is $m$-regular.  
	
	Hence the image of $\GF(R)$ via the forgetful map \(\theta\) (respectively \(\theta_{\#}\)) is as an EDT0L language over $\Sigma$ (respectively \(\Sigma\cup\{\#\}\)).
\end{corollary}

\section*{Acknowledgements}
The authors were partially supported by EPSRC Standard Grant EP/R035814/1. 

They would like to thank Angus Macintyre for his generous advice on model theory and Katrin Tent, Mike Prest and Chlo\'e Perin for helpful discussions. They would also like to thank Alex Bishop for many useful comments, including the suggestion of a stronger Proposition \ref{prop:growth2}, and Tatiana Nagnibeda and Pierre de la Harpe for advice that improved the exposition of the paper.

\bibliography{references}{}

\begin{thebibliography}{10}

\bibitem{Bodart}
P.~A. Bagnoud and C.~Bodart.
\newblock Dead ends and rationality of complete growth series.
\newblock 2022.
\newblock https://arxiv.org/abs/2210.07868v1.

\bibitem{bartholdi2010rational}
L.~Bartholdi and P.~V. Silva.
\newblock Rational subsets of groups.
\newblock 2010.
\newblock https://arxiv.org/abs/1012.1532v1.

\bibitem{Benson}
M.~Benson.
\newblock Growth series of finite extensions of {${\bf Z}^{n}$}\ are rational.
\newblock {\em Invent. Math.}, 73(2):251--269, 1983.

\bibitem{Bishop}
A.~Bishop.
\newblock Geodesic growth in virtually abelian groups.
\newblock {\em J. Algebra}, 573:760--786, 2021.

\bibitem{BishopCW}
A.~Bishop.
\newblock On groups whose cogrowth series is the diagonal of a rational series.
\newblock {\em arXiv e-prints}, 2023.
\newblock arXiv:2301.07589.

\bibitem{eqns_free_grps}
L.~Ciobanu, V.~Diekert, and M.~Elder.
\newblock Solution sets for equations over free groups are {EDT}0{L} languages.
\newblock {\em Internat. J. Algebra Comput.}, 26(5):843--886, 2016.

\bibitem{eqns_hyp_grps}
L.~Ciobanu and M.~Elder.
\newblock The complexity of solution sets to equations in hyperbolic groups.
\newblock {\em Israel J. Math.}, 245(2):869--920, 2021.

\bibitem{appl_L_systems_GT}
L.~Ciobanu, M.~Elder, and M.~Ferov.
\newblock Applications of {L} systems to group theory.
\newblock {\em Internat. J. Algebra Comput.}, 28(2):309--329, 2018.

\bibitem{GAGTA_chapter}
L.~Ciobanu and A.~Levine.
\newblock Languages, groups and equations.
\newblock {\em arXiv e-prints}, 2023.
\newblock arXiv:2303.07825.

\bibitem{VF_eqns}
V.~Diekert and M.~Elder.
\newblock Solutions to twisted word equations and equations in virtually free
  groups.
\newblock {\em Internat. J. Algebra Comput.}, 30(4):731--819, 2020.

\bibitem{RationalSets}
S.~Eilenberg and M.~P. Sch\"{u}tzenberger.
\newblock Rational sets in commutative monoids.
\newblock {\em J. Algebra}, 13:173--191, 1969.

\bibitem{Evetts}
A.~Evetts.
\newblock Rational growth in virtually abelian groups.
\newblock {\em Illinois J. Math.}, 63(4):513--549, 2019.

\bibitem{EvettsLevine}
A.~Evetts and A.~Levine.
\newblock Equations in virtually abelian groups: {L}anguages and growth.
\newblock {\em Internat. J. Algebra Comput.}, 32(3):411--442, 2022.

\bibitem{GS66}
S.~Ginsburg and E.~H. Spanier.
\newblock Semigroups, {P}resburger formulas, and languages.
\newblock {\em Pacific J. Math.}, 16:285--296, 1966.

\bibitem{Grunschlag_thesis}
Z.~Grunschlag.
\newblock {\em Algorithms in geometric group theory}.
\newblock PhD thesis, University of California, Berkeley, 1999.

\bibitem{groups_langs_aut}
D.~F. Holt, S.~Rees, and C.~E. R\"{o}ver.
\newblock {\em Groups, languages and automata}, volume~88 of {\em London
  Mathematical Society Student Texts}.
\newblock Cambridge University Press, Cambridge, 2017.

\bibitem{HrushovskiPillay}
U.~Hrushovski and A.~Pillay.
\newblock Weakly normal groups.
\newblock In {\em Logic colloquium '85 ({O}rsay, 1985)}, volume 122 of {\em
  Stud. Logic Found. Math.}, pages 233--244. North-Holland, Amsterdam, 1987.

\bibitem{EDT0L_extensions}
A.~Levine.
\newblock E{DT}0{L} solutions to equations in group extensions.
\newblock {\em J. Algebra}, 619:860--899, 2023.

\bibitem{Liardet}
F.~Liardet.
\newblock {Croissance dans les groupes virtuellement ab\'eliens}, 1997.
\newblock PhD thesis, University of Geneva.

\bibitem{LohreySurvey}
Markus Lohrey.
\newblock The rational subset membership problem for groups: a survey.
\newblock In {\em Groups {S}t {A}ndrews 2013}, volume 422 of {\em London Math.
  Soc. Lecture Note Ser.}, pages 368--389. Cambridge Univ. Press, Cambridge,
  2015.

\bibitem{NeumannShapiro95}
W.~D. Neumann and M.~Shapiro.
\newblock Automatic structures, rational growth, and geometrically finite
  hyperbolic groups.
\newblock {\em Invent. Math.}, 120(2):259--287, 1995.

\bibitem{NeumannShapiro97}
W.~D. Neumann and M.~Shapiro.
\newblock Regular geodesic normal forms in virtually abelian groups.
\newblock {\em Bull. Austral. Math. Soc.}, 55(3):517--519, 1997.

\bibitem{Prest}
M.~Prest.
\newblock {\em Model theory and modules}, volume 130 of {\em London
  Mathematical Society Lecture Note Series}.
\newblock Cambridge University Press, Cambridge, 1988.

\bibitem{TentZiegler}
K.~Tent and M.~Ziegler.
\newblock {\em A course in model theory}, volume~40 of {\em Lecture Notes in
  Logic}.
\newblock Association for Symbolic Logic, La Jolla, CA; Cambridge University
  Press, Cambridge, 2012.

\bibitem{Woods15}
K.~Woods.
\newblock Presburger arithmetic, rational generating functions, and
  quasi-polynomials.
\newblock {\em J. Symb. Log.}, 80(2):433--449, 2015.

\end{thebibliography}
\bibliographystyle{plain}

\end{document}